\documentclass[12pt]{amsart}
\usepackage{graphicx,multicol,amssymb}
\usepackage{amsthm}
\usepackage{overpic}
\usepackage{graphics}
\usepackage{amsmath}
\usepackage{lmodern}	
\usepackage{ulem}
\usepackage{amscd}
\usepackage{color}

\usepackage{latexsym}
\usepackage[all]{xy}
\begin{document}
\textwidth 5.5in
\textheight 8.3in
\evensidemargin .75in
\oddsidemargin.75in

\theoremstyle{definition}	
\newtheorem{dummy}{Dummy}[section]
\newtheorem{lem}[dummy]{Lemma}
\newtheorem{claim}[dummy]{Claim}
\newtheorem{facts}[dummy]{Fact}
\newtheorem{conj}[dummy]{Conjecture}
\newtheorem{defi}{Definition}[section]
\newtheorem{thm}[dummy]{Theorem}
\newtheorem{cor}[dummy]{Corollary}
\newtheorem{lis}{List}[section]
\newtheorem{prob}[dummy]{Problem}
\newtheorem{rmk}{Remark}[section]
\newtheorem{exm}[dummy]{Example}
\newtheorem{que}[dummy]{Question}
\newtheorem{prop}[dummy]{Proposition}
\newtheorem{clm}{Claim}[section]
\newcommand{\p}[3]{\Phi_{p,#1}^{#2}(#3)}
\def\Z{\mathbb Z}
\def\R{\mathbb R}
\def\b{\text{b}}
\def\g{\overline{g}}
\def\odots{\reflectbox{\text{$\ddots$}}}
\newcommand{\tg}{\overline{g}}
\def\ee{\epsilon_1'}
\def\ef{\epsilon_2'}
\title{Exotic elliptic surfaces without 1-handles}
\author{Motoo Tange}
\thanks{The author was partially supported by JSPS KAKENHI Grant Number 21K03216.}
\subjclass{57R65, 57R55}
\keywords{handle decomposition, elliptic surfaces, knot-surgery, log-transformation}
\address{Institute of Mathematics, University of Tsukuba,
 1-1-1 Tennodai, Tsukuba, Ibaraki 305-8571, Japan}
\email{tange@math.tsukuba.ac.jp}
\date{\today}

\begin{abstract}
In this article, we consider a sufficient condition that a knot-surgery or log-transformation of $E(n)$ admits a handle decomposition without 1-handles.
We show that if $K$ is a knot that the bridge number is $b(K)\le 9n$, then the knot-surgery $E(n)_K$ of the elliptic surface $E(n)$ admits a handle decomposition without 1-handles.
This means that if $\gcd(p,q)=1$, and $\min\{p,q\}\le 9$, then $E(1)_{p,q}$ admits a handle decomposition without 1-handles.
We also show that if $\gcd (p,q)=1$, $\min\{p,q\}\le 4$, then the double log-transformation $E(n)_{p,q}$ admits a handle decomposition without 1-handles for any positive integer $n$.
\end{abstract}
\maketitle
\section{Introduction}
\subsection{Handle decomposition of a 4-manifold.}
A long-standing interesting problem on smooth 4-manifolds is whether any simply-connected closed smooth 4-manifold admits a handle decomposition without 1-handles or 3-handles.

The following is the classical main problem about 1-handles and 3-handles of closed smooth 4-manifolds.
\begin{prob}[Kirby's problem~\cite{Kir}]
\label{mainprob}
Let $X$ be a simply-connected closed smooth 4-manifold.
\begin{itemize}
\item Does $X$ admit a handle decomposition without 1-handles?
\item Does $X$ admit a handle decomposition without 1-handles and 3-handles?
\end{itemize}
\end{prob}

Several affirmative evidences about Problem~\ref{mainprob} as below has been known.
These problems are essentially difficult to resolve.
In fact, if any homotopy $S^4$ admits a handle decomposition without 1-handles and 3-handles, the 4-dimensional smooth Poincar\'e conjecture is resolved affirmatively.

There are some simply-connected closed smooth 4-manifolds which have so complicated handle decompositions, and it does not seem that each of simply-connected closed 4-manifolds admits a handle decomposition without 1-handles and 3-handles.
\subsection{Handle decomposition of $E(n)_{K}$ or $E(n)_{p,q}$}
From the historical point of view, exotic 4-manifolds that have been known as typical examples for a long time are elliptic surface, Lefschetz fibration and their surgeries.
Let $E(n)$ be the elliptic surface with $12n$ Lefschetz singularities.
For example, double log-transformations $E(n)_{p,q}$ are homeomorphic to $E(n)$ for relatively primie integers $p,q$.
They give exotic smooth structures for $E(n)$.
Whether $E(n)_{p,q}$ admits a handle decomposition without 1-handles or 1-handles and 3-handles are not clear.

Here, we write a history on studies on 1-handles or 3-handles of $E(n)_K$, $E(n)_{p,q}$ and its blow-ups.
In the early days, Harer, Kas and Kirby in \cite{HKK} conjectured that the Dolgachev surface $E(1)_{2,3}$ requires at least one 1-handle.
% This conjecture was disproved by Yasui and Akbulut independently. 
Gompf said in \cite[p.481]{Go} that it is a good conjecture that $E(n)_{p,q}$ requires at least either a 1-handle or a 3-handle.
Harer-Kas-Kirby's conjecture in \cite{HKK} was negatively resolved by Yasui in \cite{Yasui} and Akbulut in \cite{Akb} independently.
Yasui in \cite{Yasui}, constructed handle decompositions of $E(n)_{p,q}$ for $(p,q)=(2,3),(2,5),(3,4)$ and $(4,5)$ for any positive integer $n$ without 1-handles.
Akbulut also in \cite{Akb1} proved that a knot surgery $E(1)_{K_n}$ of $E(1)$ for each of an infinite family $\{K_n\}$ of knots admits a handle decomposition without 1-handles and 3-handles.
It gives an exotic family of $E(1)$ and includes $E(1)_{2,3}$, hence, the conjecture above referred by Gompf was negatively answered in the case of $(n,p,q)=(1,2,3)$.
Sakamoto in \cite{Sa} constructed a handle decomposition of $E(1)_{2,7}$ without 1-handles.
Recently, Monden and Yabuguchi in \cite{Y} announced that $E(n)_{T_{2,q}}$ admits a handle decomposition without 1-handles and 3-handles for any odd integer $q$ and for any positive integer $n$.
This result immediately is the extionsion to $(n,p,q)=(1,2,q)$ of Akublut's result by uinsg $E(1)_{p,q}=E(1)_{T_{p,q}}$.
Kusuda in \cite{Kus} gave handle decompositions of $E(n)_{5,6}$, $E(n)_{6,7}$, $E(n)_{7,8}$ and $E(n)_{8,9}$ without 1-handles for $n\ge 4$, $n\ge 5$, $n\ge 9$, and $n\ge 24$ respectively.
Taki, in \cite{Taki}, studied an upper bound of the minimal number $l$ that $E(n)_{p,q}\#l\overline{{\mathbb C}P^2}$ admits a handle decomposition without 1-handles.
\subsection{Main results.}
The purpose of this article is to prove a wider sufficient condition that $E(n)_K$ or $E(n)_{p,q}$, which it is known that some of these manifolds are exotic to $E(n)$, has a handle decomposition without 1-handles, that is, we can eliminate all the 1-handles in the handle diagram.
Whether we can eliminate 3-handles after eliminating 1-handles is a more subtle problem, then we do not deal with the problem in this article.
The definitions of $E(n)_K$ and $E(n)_{p,q}$ will be done in Section~\ref{Prelimina}.

Here $b(K)$ is the bridge number of a knot $K$.
The following is the main theorem of this article.
\begin{thm}
\label{main}
Let $K$ be a knot in $S^3$ with $b(K)\le 9n$.
Then $E(n)_K$ admits a handle decomposition without 1-handles.
\end{thm}
This theorem is applicable to the case of double log-transformation of $E(1)$, because of $E(1)_{T_{p,q}}=E(1)_{p,q}$.
See \cite{F} for this equality.
Then we obtain the following immediately.
Here, for relatively prime positive integers $p,q$, $T_{p,q}$ stands for the right-handed $(p,q)$-torus knot.
%If $pq<0$, $T_{p,q}$ represents the left-handed $(|p|,|q|)$-torus knot.
\begin{cor}
\label{cor}
Let $p,q$ be relatively prime positive integers.
Then, if $\min\{p,q\}\le 9$, then $E(1)_{p,q}$ admits a handle decomposition without 1-handles.
\end{cor}
\begin{proof}
The bridge number of a torus knot $T_{p,q}$ is $\min\{p,q\}$.
Thus, equality $E(1)_{T_{p,q}}=E(1)_{p,q}$ and Theorem~\ref{main} make $E(1)_{p,q}$ admit a handle decomposition without 1-handles.
\end{proof}

This result can also be partially extended to more general $E(n)_{p,q}$.
\begin{thm}
\label{main3}
Let $n$ be a positive integer.
Let $p,q$ be relatively prime positive integers.
If $\min\{p,q\}\le 4$, then $E(n)_{p,q}$ admits a handle decomposition without 1-handles.
\end{thm}
Recall that the case of $n=1$ is already proven in Corollary~\ref{cor}.
This theorem also encompasses the results of Yasui in \cite{Yasui}.
To prove this proposition, we use the fact that any log-transformation is regarded as a ``twisted" knot-surgery. 
Here, we ask the following question.

%
%Moreover, for the exceptional condition in Theorem~\ref{main3}, we give a weaker condition for whether a manifold admits a handle decomposition without 1-handles by using blow-up.
%\begin{thm}
%\label{main4}
%Let $q$ be a positive integer with $\gcd(4,q)=1$.
%Then, $E(2)_{4,q}\#2\overline{{\mathbb C}P^2}$ and $E(3)_{4,q}\#\overline{{\mathbb C}P^2}$ admit handle decompositions without 1-handles.
%\end{thm}
% \subsection{Canceling 3-handles}
% Whether 3-handles are canceled instead of 1-handles is a more subtle problem than the cancellation of all 1-handles.
% We can also show the following theorem about canceling 3-handles.
% We show the following.
% \begin{thm}
% \label{main2}
% %\label{main2}
% Let $K$ be a knot in $S^3$ with $b(K)\le 6n-1$.
% Then $E(n)_K$ admits a handle decomposition without 1-handles and 3-handles.
% \end{thm}

% In the similar way to the main theorem above, we have the following corollary immediately.
% \begin{cor}
% Let $p,q$ be relatively prime positive integers with $p<q$.
% Then if $p\le 6$, then $E(1)_{p,q}$ admits handle decomposition without 1-handles and 3-handles.
% \end{cor}
% \begin{prop}
% \label{main5}
% Let $p,q$ be relatively prime positive integers with $p<q$.
% If $p\le 6$, then $E(n)_{p,q}$ admits handle decomposition without 1-handles and 3-handles.
% \end{prop}

\begin{que}
Does $E(1)_{10,11}$ admit a handle decomposition without 1-handles?
\end{que}

% \begin{que}
% Are there a simply-connected closed 4-manifold of handle decomposition which any 1-handle is not needed but a 3-handle is needed?
% \end{que}
\section*{Acknowledgements}
The first main result was essentially inspired by works of Akbulut and Yasui fifteen years ago.
However, apparently it has been less known ever for experts. 
This article was inspired by talks of Yabuguchi and Kusuda in the workshop ``Four Dimensional Topology" on October in 2024 held at Osaka.
The author is grateful for their results so much.
Also, he thanks Kouichi Yasui for giving some helpful advice for an early draft.

\section{Preliminaries}
\label{Prelimina}
\subsection{Knot-surgery}
\label{knotsurgery}
Let $X$ be a 4-manifold in which the trivial tubular neighborhood of a torus $T$ is contained.
Hence, we can take a diffeomorphism $\nu(T)\cong T\times D^2$.
Here, $\nu(A)$ is a tubular neighborhood of a submanifold $A$.
Let $K\subset S^3$ be a knot.
Then, {\it knot-surgery} is defined by the following surgery:
$$X_K:=(X-\nu(T))\cup_\phi (S^3-\nu(K))\times S^1.$$
The gluing map $\phi:\partial\nu(K)\times S^1\to T\times \partial D^2$ satisfies $\phi(m)=\lambda_1$, $\phi(s)=\lambda_2$, $\phi(l)=d$,
where $m,l\subset \partial\nu(K)\times \{\text{pt}\}$ are the meridian and longitude of $K$ respectively, $s$ is $\{\text{pt}\}\times S^1$, $d$ is the meridian circle $\{\text{pt}\}\times \partial D^2$ of $T$, and $\lambda_1, \lambda_2\subset T\times \{\text{pt}\}$ are two simple closed circles generating $H_1(T)$.

\subsection{Bridge presentation of knot}
Let $K$ be a knot in $S^3$.
If there exists a decomposition $S^3=B^3_0\cup B^3_1$ with $B^3_0\cap B^3_1=\partial B_0^3=\partial B^3_1$, such that $B_i^3$ is the 3-ball and the intersection $K\cap B_i=K_i$ is a set of $n$ boundary-parallel arcs in the 3-ball.
This is called an {\it $n$-bridge presentation} of $K$.
We define the bridge number
$$b(K)=\min\{n|K\text{ admits an $n$-bridge presentation}\}.$$
A knot with $b(K)=1$ is the trivial knot.
For example, a knot with $b(K)=2$ is called a 2-bridge knot.

It is easy to show that an $n$-bridge presentation of $K$ has a normal form as in Figure~\ref{normalformnbridge}.
\begin{figure}[thbp]
\centering
\begin{overpic}[
%grid,tics=10
]{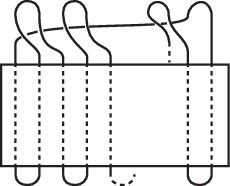}
\put(58,28){$B$}
\end{overpic}
\caption{A normal form of an $n$-bridge knot.
$B$ is a pure braid.}  
\label{normalformnbridge}
\end{figure}

\begin{facts}
Let $K$ be a knot with an $n$-bridge presentation.
Then, we can find a bridge presentation as Figure ~\ref{normalformnbridge} using a pure braid $B$ of a $2n$-string.
\end{facts}

The set of generators of the pure braid group $PB_n$ with an $n$-string is well-known.
\begin{facts}
The pure braid group $PB_n$ has the following set of generators $\{\mathcal{T}_{i,j}\mid 1\le i<j\le n\}$ as in Figure~\ref{purebraidgenerator}.
\end{facts}

\begin{figure}[thbp]
\centering
\begin{overpic}[
%grid,tics=10
]{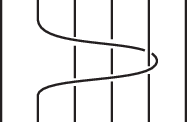}
\put(25,3){$i$}
\put(70,3){$j$}
\end{overpic}
\caption{A generator element $\mathcal{T}_{i,j}$ of $PB_n$.}  
\label{purebraidgenerator}
\end{figure}

\subsection{Handle decomposition of knot-surgery}
\label{handledecknotsurgery}
In this section, we construct a handle decomposition of knot-surgery.
This is due to \cite{A,T1,T2}.
Let $X$ be a 4-manifold which $T^2\times D^2$ is contained.
We may assume that we regard the handle decomposition of $X$ as a decomposition constructed by attaching several handles over $T^2\times D^2=h^0\cup h^1_1\cup h_2^1\cup h^2$.
Here, $h^k$ stands for a $k$-handle.
The handle decomposition of $T^2\times D^2$ is as in Figure~\ref{T2D2}.

\begin{figure}[thbp]
\centering
\begin{overpic}[
%grid,tics=10
]{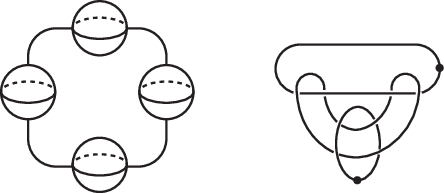}
\put(3,3){$0$}
\put(50,20){$\cong$}
\end{overpic}
\caption{A handle decomposition of $T^2\times D^2$.}
\label{T2D2}
\end{figure}

First, we insert 2-/3-canceling pairs and 1-/2-canceling pairs in the positions in the second and third pictures in Figure~\ref{handlediffeo}.
We deform the diagram as in Figure~\ref{handlediffeo}, which presents a deformation in the case of two pairs of 2-/3-canceling handles.
In general, similarly, we can also deform the diagram in the case where the number of 2-/3-canceling pairs is $n-1$.

Second, we use an $n$-bridge presentation of $K$ in a normal form using a $2n$-pure braid $B$ as shown in Figure~\ref{normalformnbridge}.
We remove $T^2\times D^2$ from $X$.
According to the presentation $B=\mathcal{T}_{i_1,j_1}\mathcal{T}_{i_2,j_2}\cdots, \mathcal{T}_{i_r,j_r}$, we deform the surgery diagram of the boundary.
For example, $\mathcal{T}_{1,2}$ and $\mathcal{T}_{3,4}$ are isotopies by twisting.
As another example, the deformation in the case of $\mathcal{T}_{2,5}$ is shown in Figure~\ref{3bridgeknot}.
These are the boundary diffeomorphisms over $\partial (S^3-\nu(K))\times S^1$.
Gluing $(S^3-\nu(K))\times S^1$ to the boundary, we obtain the diagram of $X_K$.
In the last diagram in Figure~\ref{handlediffeo}, we call a 0-framed 2-handle located in the center of the diagram a {\it centered 0-framed 2-handle} in the diagram of $(S^3-\nu(K))\times S^1$.
\begin{figure}[thbp]
\centering
\begin{overpic}[
%grid,tics=10
]{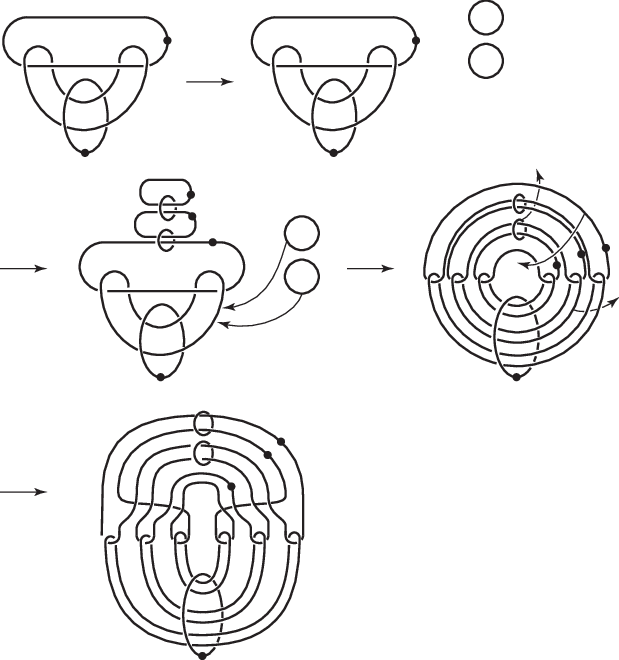}

\put(70,80){$\cup$3-handles}

\put(30,44){$\cup$3-handles}

\put(88,44){$\cup$3-handles}
\put(44,4){$\cup$3-handles}

\end{overpic}
\caption{A deformation of the handle diagram of $T^2\times D^2$. All the components with no dots are 0-framed 2-handles.}
\label{handlediffeo}
\end{figure}

\begin{figure}[thbp]
\centering
\begin{overpic}[
%grid,tics=10
]{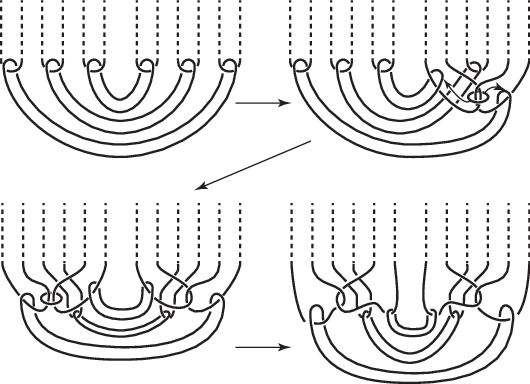}
%\put(58,28){$PB$}
\end{overpic}
\caption{A deformation of diagram for a generator element $\mathcal{T}_{2,5}$ in $PB_6$}  
\label{3bridgeknot}
\end{figure}
\subsection{Log-transformation}
Here, we give a short review of the log-transformation.
Let $T^2\subset X$ be an embedded torus with the trivial normal bundle.
The boundary of the tubular neighborhood $\nu(T^2)\cong T^2\times D^2$ of $T^2$ is a 3-torus $T^3$.
The generator set of $H_1(\partial (T^2\times D^2))$ is represented by $\{d,\lambda_1,\lambda_2\}$, where $d=\{\text{pt}\}\times\partial D^2 $, and $\{\lambda_1,\lambda_2\}$ presents a generator set of $H_1(T^2\times \{\text{pt}\})$.
Then we define a map $\phi:T^2\times \partial D^2\to T^2\times \partial D^2$ as follows:
$$\phi(d)=p'\cdot\gamma+p\cdot d,$$
where $p$ and $p'$ are relatively prime and $\gamma$ is some primitive element in $H_1$ of the fiber torus with $\gamma=b\cdot\lambda_1+c\cdot \lambda_2$ for some integers $b,c$.
In this paper, for two simple closed (homologically non-trivial) curves $x_1,x_2$ in a 2-torus and for $a_1, a_2\in {\mathbb Z}$, the notation $a_1\cdot x_1+a_2\cdot x_2$ stands for an isotopy class of a simple closed curve representing the element $a_1[x_1]+a_2[x_2]$ in $H_1$ of the torus.
If the circle is one component, then $\gcd(a_1,a_2)$ must be 1.

Then the {\it log-transformation} along $T^2$ is the following surgery: 
$$X_{\gamma,p',p}:=(X-\nu(T^2))\cup_\psi T^2\times D^2.$$
Here $p$ is called a {\it multiplicity} of log-transformation.

We assume $\pi_1(X-\nu(T^2))=e$.
The diffeomorphism class of the log-transformation depends only on the multiplicity, because of this $\pi_1$-condition.
We denote the log-transformation by $X_p$.
When we perform the log-transformation along two parallel tori with multiplicities $p,q$ respectively, we say the surgery {\it double log-transformation}.
We then denote the result as $X_{p,q}$.

Here we state a relation between the double log-transformation with multiplicities $p,q$ and the knot-surgery of $T_{p,q}$.
Let $T\subset X$ be an embedded torus with the trivial normal bundle.
Let $d,\lambda_1,\lambda_2,m,l,s$ be the same as the things defined in Section~\ref{knotsurgery}.
\begin{prop}
\label{logknot}
The log-transformation $X_{p,q}$ is diffeomorphic to $(X-\nu(T))\cup_{\phi'} (S^3-\nu(T_{p,q}))\times S^1$.
The gluing map $\phi':\partial (S^3-\nu(T_{p,q}))\times S^1\to \partial \nu(T)$ satisfies $\phi'(m)=d$, $\phi'(l)=\lambda_1-pq\cdot d$, and $\phi'(s)=\lambda_2$.
\end{prop}
\begin{proof}
The manifold $S^3-\nu(T_{p,q})$ is diffeomorphic to the twice Dehn surgery of $S^1\times D^2$ along two parallel curves $f_1,f_2$ parallel to $\lambda_1=S^1\times \{\text{pt}\}$ as shown in Figure~\ref{Seifert}.
The slopes of the Dehn surgeries are $p/u$ and $q/v$ where $p,q$ are degrees about the meridians of $f_1,f_2$ and satisfy $pv+qs=1$.
A slope $r/s$ of $T_{p,q}$ corresponds to a slope $r/s-pq$ of the dotted circle in the Seifert structure in the middle picture of Figure~\ref{Seifert}.
By taking the product of $S^1$, $(S^3-\nu(T_{p,q}))\times S^1$ is diffeomorphic to $(T^2\times D^2)_{(\lambda_1,u,p),(\lambda_1,v,q)}$.
Here, the $S^1$-direction is $\lambda_2$.

The meridian circle of $T_{p,q}$, i.e., $\infty$-slope in $\partial \nu(T_{p,q})$ is mapped to the $\infty$-slope of the dotted circle in Figure~\ref{Seifert}.
The longitude circle ($0$-slope) of $T_{p,q}$ is mapped to a circle representing $\lambda_1-pq\cdot d$ in $\partial \nu(T)$.
The $S^1$-direction in $(S^3-\nu(T_{p,q}))\times S^1$ is mapped to $\lambda_2$.

Hence, the gluing map $\phi':\partial(S^3- \nu(T_{p,q}))\times S^1\to \partial \nu(T)$ satisfies $\phi'(m)=d$, $\phi'(l)=\lambda_1-pq\cdot d$, and $\phi'(s)=\lambda_2$.
Then we have the following.
\begin{eqnarray*}
X_{p,q}&=&(X-\nu(T))\cup (D^2\times T^2)_{p,q}\\
&\cong& (X-\nu(T))\cup_{\phi'} (S^3-\nu(T_{p,q}))\times S^1.
\end{eqnarray*}
\end{proof}

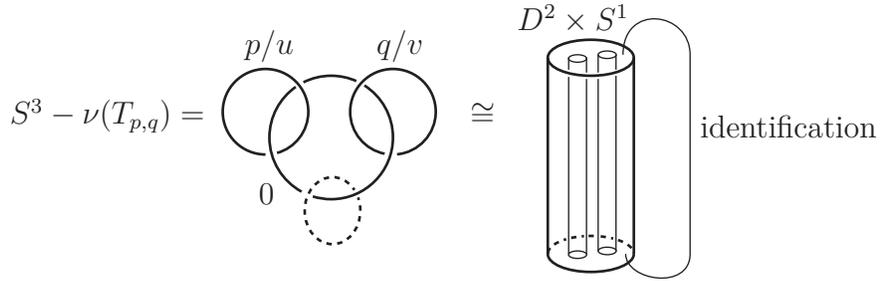
\begin{figure}[thbp]
\centering
\begin{overpic}[
%grid,tics=10
]{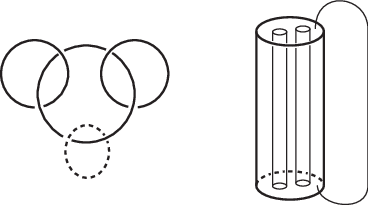}

\put(-45,33){$S^3-\nu(T_{p,q})=$}
\put(5,48){$p/u$}
\put(33,48){$q/v$}
\put(8,16){$0$}
\put(53,33){$\cong$}
\put(63,53){$D^2\times S^1$}
\put(102,30){identification}
\end{overpic}
\caption{A diffeomorphism from $S^3-\nu(T_{p,q})$ to a double log-transformation along two curves.}
\label{Seifert}
\end{figure}

% We note that the right-hand side is not diffeomorphic to $X_{T_{p,q}}$ for $n>1$, in general.
% Thus, a handle decomposition of $E(n)_{p,q}$ admits $(S^3-\nu(T_{p,q}))\cup\text{(2-handles)}$.
% These 2-handles are attached in the same way as (a) in Figure~\ref{knotsurgerydef}.
% Thus, if $p\le 9$, then all the 1-handles in $(S^3-\nu(T_{p,q}))\times S^1$ are removed.
% Over this manifold, we attach $E(n)^\nu(T)-\text{(2-handles)}$ using some gluing map.
% Since this part does not include 1-handles, we obtain a handle decomposition of $E(n)_{p,q}$ without 1-handles.

\section{Proofs}
\subsection{The global monodromy of $E(n)$}
\label{Canceling 1-handle}
Let $E(n)$ be the elliptic surface with $12n$ Lefschetz singularities.
% For simplicity, we consider the case of $n=1$.
We divide $E(n)$ a neighborhood of general fiber $T$ and the complements as in $(E(n)-\nu(T))\cup T\times D^2$.
We perform the knot-surgery in $E(n)-\nu(T)$
$$E(n)_K\supset (E(n)-\nu(T))_K$$

Let two elements $a,b$ be $\begin{pmatrix}1&1\\0&1\end{pmatrix}$, and $\begin{pmatrix}1&0\\-1&1\end{pmatrix}\in SL(2,{\mathbb Z})$.
The global monodromy of $E(n)$ is $(ab)^{6n}$ and can be deformed as follows
\begin{eqnarray}
\nonumber(ab)^{6n}&=&(abababababab)^n=(a(aba)a(aba)a(aba))^n\\
&=&(a^2ba^3ba^3ba)^n\sim(a^3ba^3ba^3b)^n.\label{monodromy}
\end{eqnarray}

\subsection{Proof of Theorem~\ref{main}.}
Let $K$ be a knot with $b(K)\le 9n$.
Since $E(n)$ has $12n$ Lefschetz singularities, the handle decomposition of $E(n)$ consists of the union of $T\times D^2\cong h^0\cup h^1_1\cup h^1_2\cup h^2$, $12n$ $(-1)$-framed 2-handles and $h^2\cup h^3_1\cup h^3_2\cup h^4$.
We obtain a handle decomposition around a general fiber $T$ as in Figure~\ref{knotsurgerydef} (a).
For simplicity, here it is the case of $n=1$.
From the computation of monodromy (\ref{monodromy}), the parallel $9n$ $-1$-framed 2-handles in the figure correspond to the vanishing cycles for the element $a$.
Another $-1$-framed 2-handle corresponds to the vanishing cycle for $b$.
By removing $T\times D^2$ and gluing $(S^3-\nu(K))\times S^1$ as shown in Figure~\ref{knotsurgerydef} (b), 
the meridian circle $d$ of $T\times D^2$, through the boundary diffeomorphism, is mapped to the meridian of a 1-handle of $(S^3-\nu(K))\times S^1$ as illustrated in Figure~\ref{knotsurgerydef}.
Here, sliding some $-1$-framed 2-handles over 0-framed 2-handles we can move them to each of the meridians of other 1-handles.
We obtain the handle diagram in Figure~\ref{knotsurgerydef} (c).
Since we can give a handle decomposition having $b(K)+1$ 1-handles in $(S^3-\nu(K))\times S^1$, all 1-handles in the picture can be canceled from the condition $b(K)\le 9n$.
Since $E(n)-\nu(T)$ has no 1-handles, $E(n)_K$ admits a handle decomposition without 1-handles.\hfill$\Box$
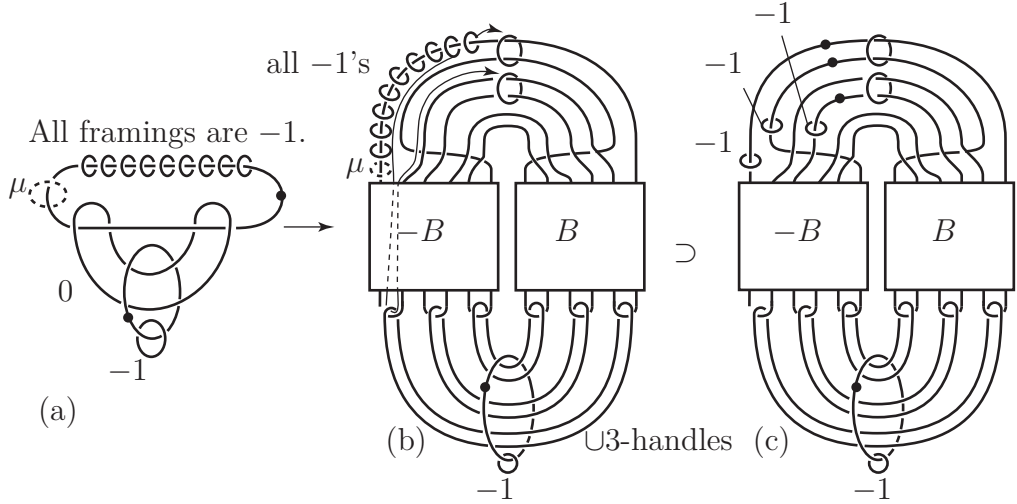
\begin{figure}[thbp]
\centering
\begin{overpic}[
%grid,tics=10
]{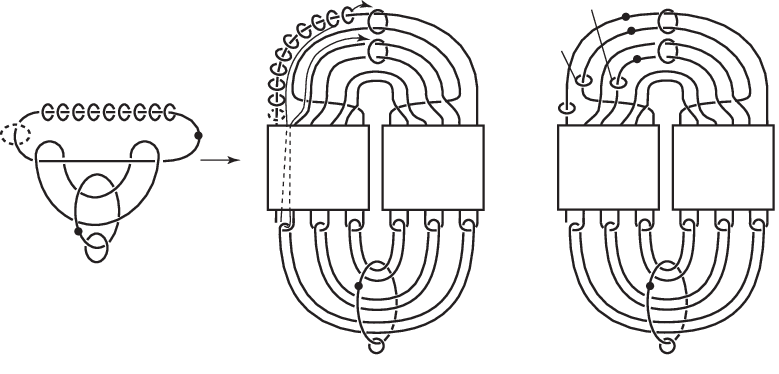}

\put(0,36){All framings are $-1$.}
\put(-2,31){$\mu$}
\put(3,20){$0$}
\put(1,8){(a)}

\put(24,43){all $-1$'s}
\put(37,26){$-B$}
\put(53,26){$B$}
\put(36,5){(b)}
\put(56,5){$\cup$3-handles}
\put(65,24){$\supset$}
\put(32,33){$\mu$}

\put(67,35){$-1$}
\put(68,43){$-1$}
\put(73,48){$-1$}
\put(75,26){$-B$}
\put(91,26){$B$}
\put(8,12){$-1$}
\put(45,0){$-1$}
\put(83,0){$-1$}
%\put(95,5){$\cup$3-handles}
\put(73,5){(c)}
\end{overpic}
\caption{A knot-surgery near a regular fiber with $9+1$ vanishing cycles. Here $B$ is a pure braid.}
\label{knotsurgerydef}
\end{figure}

Since $b(T_{p,q})=\min\{p,q\}$, the condition $\min\{p,q\}\le 9n$ gives a handle decomposition of $E(n)_{T_{p,q}}$ without 1-handles.
%\begin{exm}
%In particular since $T_{p,q}$ with $p\le 9$ is $\b\le 9$, $E(1)_{T_{p,q}}$ admits handle decomposition without no 1-handles.
%\end{exm}
% In the same method, we easily obtain the following.
% \begin{cor}
% Let $K$ be a knot with $b(K)\le 9n$, then $E(n)_K$ admit handle decomposition without 1-handles.
% \end{cor}
\subsection{Proof of Theorem~\ref{main3}.}
Let $p,q$ be relatively prime positive integers.
We decompose $E(n)_{p,q}=((S^3-\nu(T_{p,q}))\times S^1)\cup_{\phi'} (E(n)-\nu(T))$ according to Proposition~\ref{logknot}.
The 1-handles of $E(n)_{p,q}$ are induced by the 1-handles of $(S^3-\nu(T_{p,q}))\times S^1$.
Over the meridian of the 1-handle of the $S^1$-direction, a $-1$-framed 2-handle is attached from $E(n)-\nu(T)$ through $\phi'$.
Hence, this 1-handle is canceled.
Over the 1-handle for the meridian $m$ in $(S^3-\nu(T_{p,q}))\times S^1$, a section $E(n)-\nu(T)$ is attached.
The section gives a $-n$-framed 2-handle $h_0$ in $E(n)-\nu(T)$.

We create a 0-framed 2-handle $h_1$ as a 2-/3-canceling pair in the diagram.
Slide the 0-framed 2-handle $h_0$ over the $-n$-framed 2-handle as shown in Figure~\ref{handlecal3}.
Then we obtain two meridians $h_0,h_1$ with linking $-n$.

Here, we use the parallel $9n$ 2-handles attached over $\lambda_1$.
According to the process as illustrated in Figure~\ref{chain}, we can construct a chain type link with all framings $-2$ and with the length $9n-1$.
We call this link a {\it $-2$-chain with length $9n-1$}.
Removing the $(n+2)$-nd component in the $-2$-chain, we get a $-2$-chain with length $n+1$.

Appropriately sliding the 2-component 2-handles $(h_0,h_1)$ with linking $-n$ over the $-2$-chain with length $n+1$, we can untie the $-n$-linking as shown in Figure~\ref{handlecal2}.
In this picture we illustrate the case of $n=1,2,3$ only.
The general cases are proven after this proof. 
Then, the framings of the two 2-handles $(h_0',h_1')$ become 
$$\begin{cases}(-2n-1,-2n-1)&n\text{: odd}\\(-2n,-2n-2)&n\text{: even.}\end{cases}$$
After this, removing 2-handles for the $-2$-chain with length $n+1$ from the diagram, we obtain two meridian 2-handles. 
We slide one of the 2-handles over the centered 0-framed 2-handle to realize the canceling position of two 1-handles.
Recall the definition of the centered 0-framed 2-handle in Section~\ref{handledecknotsurgery}.
Therefore, we can cancel all 1-handles of $E(n)_{2,q}$.

%As shown in Figure~\ref{handlecal}, in the handle diagram of $(S^3-\nu(T_{p,q}))\times S^1$, 
Next, we create two 0-framed 2-handles $h_2, h_3$ as two 2-/3-canceling pairs in the diagram.
Slide the 0-framed 2-handles over $h_0'$ and $h_1'$ respectively in the similar manner.
Then we obtain two pairs of meridian 2-handles with linking $(-2n-1,-2n-1)$, or $(-2n,-2n-2)$.
We untie the linking by using the rest of $-2$-chain with length $8n-3$.
To unlink the two linking, we need a $-2$-chain with length $2n+2$ and $2n+2$ or $2n+1$ and $2n+3$.
Hence, totally, we need a $-2$-chain with length
$$(2n+2)+1+(2n+2)=(2n+1)+1+(2n+3)=4n+5.$$
Since $4n+5\le 8n-3$ holds for any integer $n\ge2$, this is realized.
We obtain four meidian 2-handles with framings
$$
\begin{cases}
(-4n-3,-4n-3,-4n-3,-4n-3)&n\text{: odd}\\
(-4n,-4n-2,-4n-4,-4n-6)&n\text{: even}
\end{cases}$$

%Next, we create a 0-framed 2-handle $h_2$ as a 2-/3-canceling pair in the diagram.
%Slide the 0-framed 2-handle over $h_0'$ in the similar manner.
%Then we obtain two meridian 2-handles with linking $-2n-1$, or $-2n$.
%Let $-N$ denote $-2n-1$, or $-2n$.
%We untie the linking by using the rest of $-2$-chain with length $8n-3$.
%The we obtain a meridian 2-handle $h_2'$.
%Here we remove the ($N+2$)-nd component in the rest of $-2$-chain, then we use a $-2$-chain with length $N+1$.
%Then the framings of the two 2-handles become
%$$\begin{cases}
%(-2N-1,-2N-3)&n\text{: odd}\\
%(-2N,-2N-2)&n\text{: even}.
%\end{cases}$$
%Then we obtain three meridian 2-handles $(h_0',h_1',h_2')$.
%The framings are
%$$\begin{cases}
%(-4n-3,-2n-1,-4n-3)&n\text{: odd}\\
%(-4n,-2n-2,-4n-2)&n\text{: even}.
%\end{cases}$$
%Using the centered 0-framed 2-handles, we distribute the three components to each meridian of 1-handles.
%These are canceling pairs.
%Therefore, we can cancel all 1-handles of $E(n)_{3,q}$.

%%Furthermore for the 2-handle $h_1'$, we carry out the same process.
%%That is, create a 0-framed 2-handle $h_3$, slide $h_3$ over $h_1'$, we untie the linking by using a $-2$-chain.
%%Then we obtain two meridians $(h_1'',h_3')$.
%%We reduced $5n+8$-component from the $9n-1$-component $-2$-chain.
%%Here $5n+8\le 9n-1$ holds for $n>2$.
%%Even in the case of $n=2$, the last removing $-2$-framing is not needed, that is,
%%$$9n-1=17=3+1+5+1+7$$
%%holds.
%Then, we realize four meridian 2-handles of the 1-handles.
Using the centered 0-framed 2-handles, we distribute the four components to each meridian of 1-handles.
These are canceling pairs.
Therefore, we can cancel all 1-handles of $E(n)_{3,q}$ or $E(n)_{4,q}$.

\begin{figure}[htbp]
\centering
\begin{overpic}[
%grid,tics=10
]{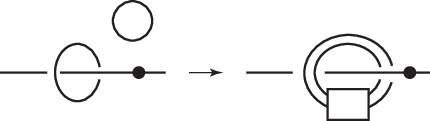}

\put(20,22){$0$}
\put(3,13){$-n$}

\put(76,2){$-n$}
\put(60,12){$-n$}
\put(79,21){$-n$}
\end{overpic}
\caption{A handle slide. The box including $-n$ stands for the $-n$-full twist.}
\label{handlecal3}
\end{figure}
\begin{figure}[htbp]
\centering
\begin{overpic}[
%grid,tics=10
]{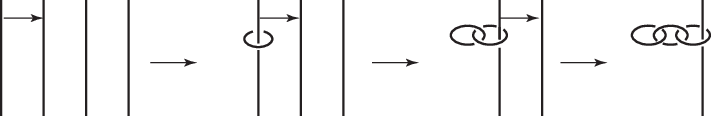}

\put(-5,1){$-1$}
\put(1,1){$-1$}
\put(7,1){$-1$}
\put(13,1){$-1$}

\put(29,10){$-2$}
\put(31,1){$-1$}
\put(37,1){$-1$}
\put(43,1){$-1$}

\put(60,13){$-2$}
\put(64,7){$-2$}
\put(65,1){$-1$}
\put(71,1){$-1$}

\put(86,14){$-2$}
\put(91,7){$-2$}
\put(93,14){$-2$}
\put(94,1){$-1$}
\end{overpic}
\caption{Handle slides to construct a chain of framed link.}
\label{chain}
\end{figure}

\begin{figure}[htbp]
\centering
\begin{overpic}[
%grid,tics=10
]{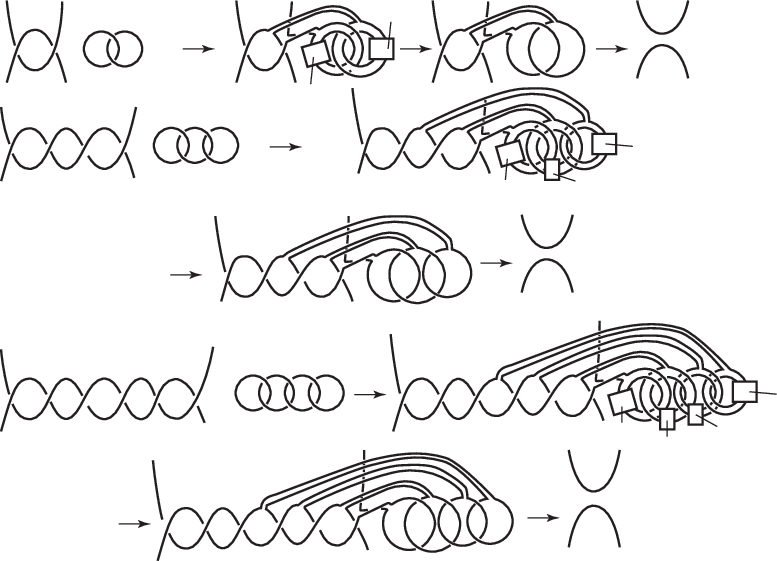}

\put(3,69){$-1$}
\put(3,60){$-1$}
\put(10,69){$-2$}
\put(15,69){$-2$}

\put(31,69){$-3$}
\put(31,60){$-3$}

\put(37,59){$-2$}
\put(48,70){$-2$}

\put(56,69){$-3$}
\put(56,60){$-3$}

\put(88,69){$-3$}
\put(88,64){$-3$}

\put(1,57){$-2$}
\put(1,47){$-2$}
\put(20,57){$-2$}
\put(24,57){$-2$}
\put(28,57){$-2$}

\put(46,57){$-4$}
\put(46,47){$-6$}
\put(63,47){$-2$}
\put(74,48){$-2$}
\put(82,53){$-2$}

\put(29,40){$-4$}
\put(29,31){$-6$}

\put(74,41){$-6$}
\put(74,36){$-4$}

\put(1,25){$-3$}
\put(1,15){$-3$}
\put(28,25){$-2$}
\put(32,25){$-2$}
\put(36,25){$-2$}
\put(41,25){$-2$}

\put(51,25){$-7$}
\put(57,25){$-7$}
\put(78,15){$-2$}
\put(83,13){$-2$}
\put(91,15){$-2$}
\put(100,20){$-2$}

\put(21,8){$-7$}
\put(27,8){$-7$}

\put(80,9){$-7$}
\put(80,3){$-7$}
\end{overpic}
\caption{Unlinking processes of $n=1$, $2$, $3$ using a $-2$-chain with length $n+1$.}
\label{handlecal2}
\end{figure}
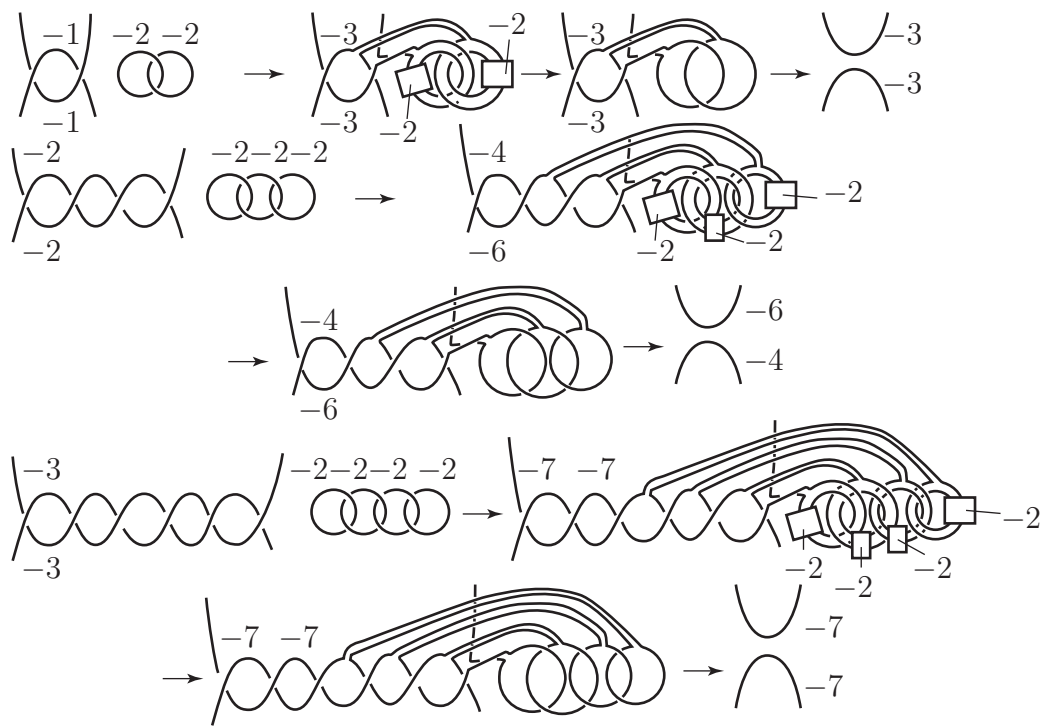
\hfill$\Box$\\

Finally, we explain the unlinking process in the previous proof.
For a chain of ($n+1$)-component unknots, we perform band moves by $n+1$ parallel bands connecting maximal points of $-n$-linking.
Then we obtain the first poicture in Figure~\ref{handlecal4}.
We have only to show that this picture is isotopic to unlinked two arcs.
The two indicated crossings in the picture can be simultaneously changed by an isotopy.
Then, we obtain the second picuture.
Shrinking the right most component, we obtain the third picture.
This picture is the one replaced by $n\to n-1$ for the first picture.
In the same way, iterating this process, we obtain the picture which is isotopic to the picture having $-2$-linking and 3 components as illustrated in the last row. 
This picture is isotopic to unlinked arcs as proven in the previous proof.
\begin{figure}[htbp]
\centering
\begin{overpic}[
%grid,tics=5
]{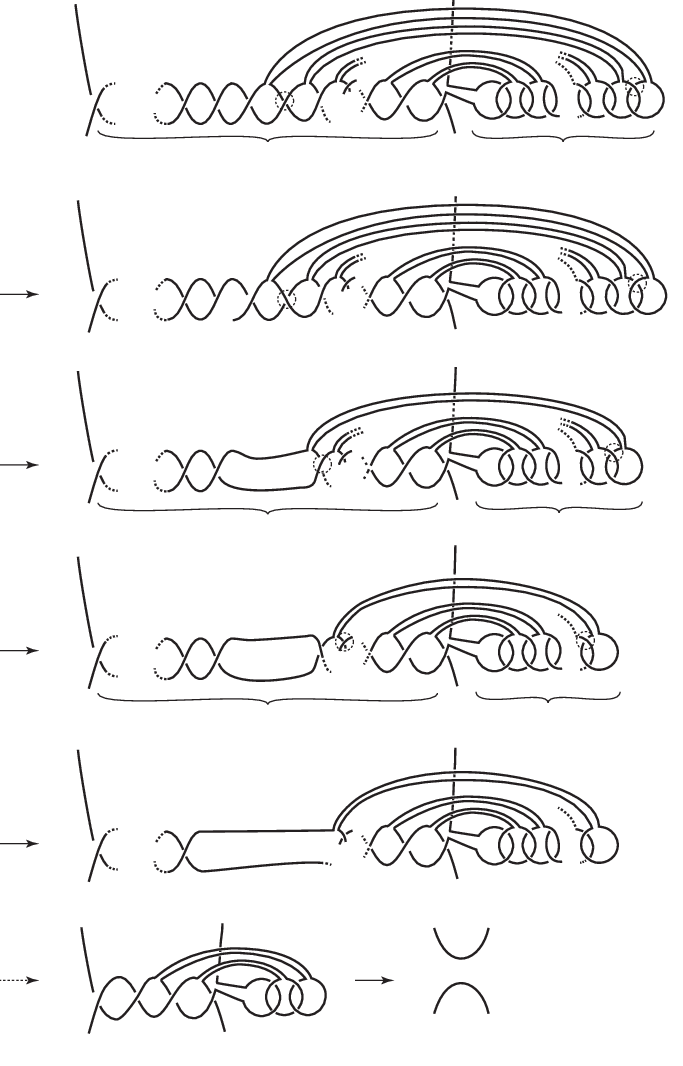}

\put(20,84.5){$-n$-linking.}
\put(45,84.5){$n\!+\!1$ components}

\put(20,50){$-n+1$-linking.}
\put(44,50){$n$ components.}

\put(20,32){$-n\!+\!2$-linking.}
\put(43,32){$n\!-\!1$\! components.}

\put(7,1){$-2$-linking.}
\put(22,1){$3$ components.}
\end{overpic}
\caption{A general unlinking process.}
\label{handlecal4}
\end{figure}
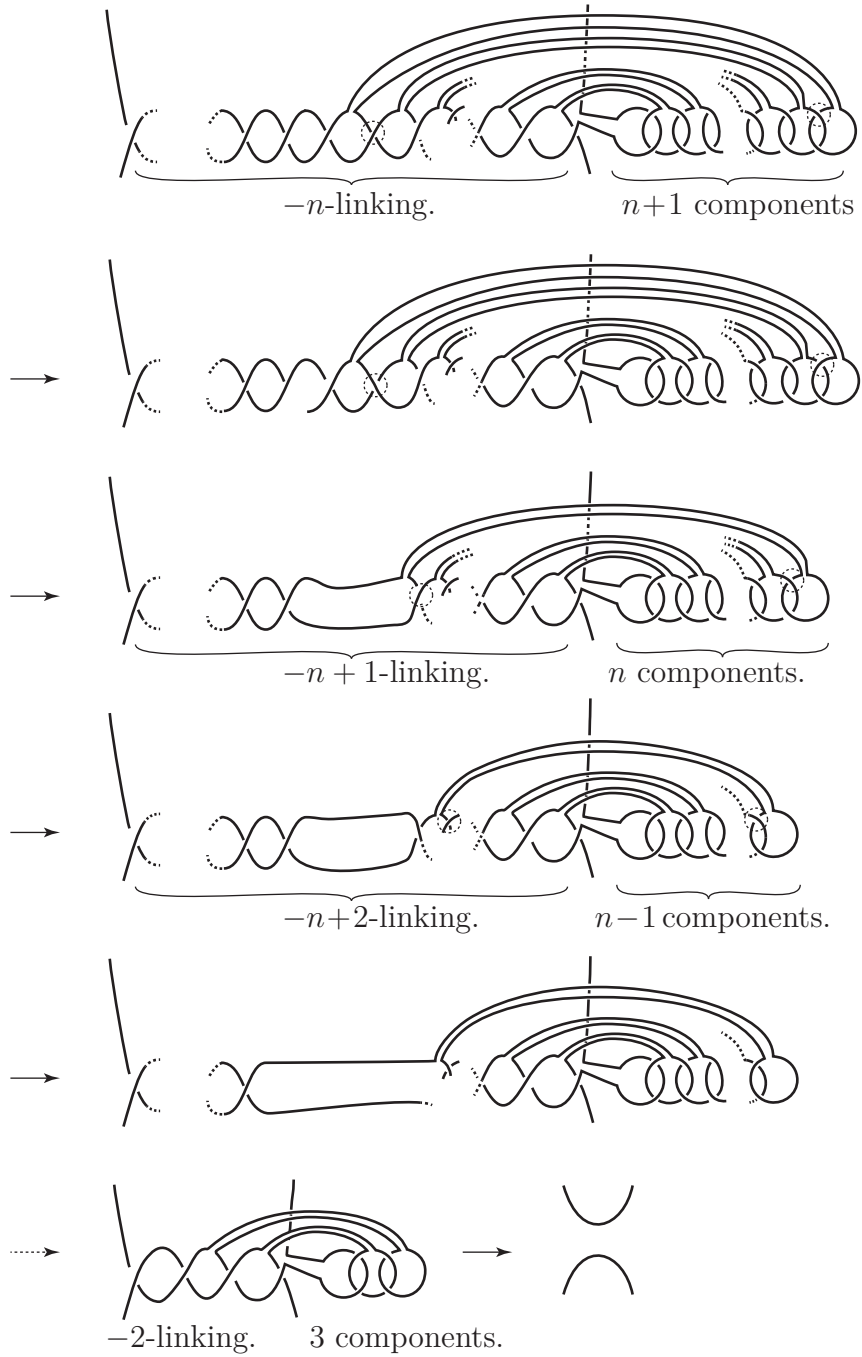

\noindent
Moto-o Tange\\
University of Tsukuba, \\
Ibaraki 305-8502, Japan. \\
tange@math.tsukuba.ac.jp


\begin{thebibliography}{10}
\bibitem{Akb} S. Akbulut, {\it The Dolgachev surface. Disproving the Harer-Kas-Kirby conjecture}, Comment. Math. Helv. 87 (2012), no. 1, 187--241. 
\bibitem{Akb1} S. Akbulut, {\it An infinite family of exotic Dolgachev surfaces without 1- and 3- handles}, Journal of G\"okova Geometry Topology
Volume 3 (2009) 22--43. 
\bibitem{A} S. Akbulut, {\it Variations on Fintushel-Stern knot surgery on 4-manifolds}, Turkish J. Math. 26 (2002), no. 1, 81-92.
\bibitem{F} R. Fintushel and R. Stern, {\it Six lectures on 4-manifolds}, In Low dimensional topology, IAS/Park City Math. Ser. 15, Amer. Math. Soc., Providence, RI, 2009, 265--315.
\bibitem{Go} R. Gompf, {\it Nuclei of elliptic surfaces}, opology 30 (1991), no. 3, 479--511.
\bibitem{HKK} J. Harer, A. Kas, and R. Kirby, {\it Handlebody decompositions of complex surfaces}, Mem. Amer. Math. Soc. 62 (1986), no. 350, iv+102 pp.
\bibitem{Kir} R. Kirby, {\it Problems in Low-Dimensional Topology}, in Geometric Topology (W. Kazez ed.), AMS/IP.
Stud. Adv Math. vol. 2.2, Amer. Math. Soc. (1997), 35--473.
\bibitem{Kus} D. Kusuda, {\it On elliptic surfaces which have no 1-handles}, arXiv:2410.16900.
\bibitem{Y} N. Monden, and R. Yabuguchi, {\it Knot surgered elliptic surfaces without 1- and 3-handles}, workshop ``Four Dimensional Topology" in 2024.
\bibitem{Sa} R. Sakamoto, {\it A geometrically simply connected elliptic surface}, Master thesis, Osaka University (2023).
\bibitem{Taki} K. Taki, {\it Geometrically simply connectedness of elliptic surfaces and blow-ups}, Master thesis, Osaka University (2024).
\bibitem{T1} M. Tange, {\it On the diffeomorphisms for Akbulut's knot concordance surgery}, J. Knot Theory Ramifications 14 (2005), no. 5, 539--563.
\bibitem{T2} M. Tange, {\it The link surgery of  $S^2\times S^2$  and Scharlemann's manifolds}, Hiroshima Math. J. 44 (2014), no. 1, 35--62.
\bibitem{Yasui} K. Yasui, {\it Elliptic surfaces without 1-handles}, J. Topol. 1 (2008), no. 4, 857--878.
\end{thebibliography}
\end{document}